\documentclass[11pt]{amsart}
\usepackage{amsmath, amsfonts, amsthm, amssymb, multicol, mathtools, dsfont, verbatim}
\usepackage{graphicx}
\usepackage{float, hyperref}
\usepackage{scalerel,stackengine, subcaption}
\usepackage[usenames,dvipsnames]{xcolor}
\usepackage{enumitem}
\usepackage{pgfplots}
\usepgfplotslibrary{fillbetween}
\pgfplotsset{width=10cm,compat=1.9}
\usepackage[square,sort,comma,numbers]{natbib}
\allowdisplaybreaks
\newcommand\numberthis{\addtocounter{equation}{1}\tag{\theequation}}

\usepackage{fancyhdr}
 
\numberwithin{equation}{section}

\newtheorem{thm}{Theorem}[section]

\newtheorem{cor}[thm]{Corollary}

\newtheorem{prop}[thm]{Proposition}
\newtheorem{conj}[thm]{Conjecture}

\newtheorem{ques}[thm]{Question}

\renewcommand{\epsilon}{\varepsilon}
\newcommand{\eps}{\varepsilon}

\newcommand{\rd}{\mathbb{R}^d}

\renewcommand{\geq}{\geqslant}
\renewcommand{\leq}{\leqslant}

\newcommand{\ubd}{\overline{\dim}_{\textup{B}}}
\newcommand{\ubdp}{\overline{\dim}^s_{\textup{B}}}
\newcommand{\lbd}{\underline{\dim}_{\textup{B}}}
\newcommand{\hd}{\dim_{\textup{H}}}
\newcommand{\pd}{\dim_{\textup{P}}}

\newcommand{\fs}{\dim^\theta_{\mathrm{F}}}
\newcommand{\efs}{\textup{E} \hspace{-0.5mm}\dim^\theta_{\mathrm{F}}}
\newcommand{\efd}{\textup{E} \hspace{-0.5mm}\dim_{\mathrm{F}}}
\newcommand{\fd}{\dim_{\mathrm{F}}}
\newcommand{\sd}{\dim_{\mathrm{S}}}

\makeatletter
\def\@setauthors{%
  \begingroup
  \def\thanks{\protect\thanks@warning}%
  \trivlist
  \centering\footnotesize \@topsep30\p@\relax
  \advance\@topsep by -\baselineskip
  \item\relax
  \author@andify\authors
  \def\\{\protect\linebreak}

  \normalsize\lowercase{\authors}%
  
	\ifx\@empty\contribs
  \else
    ,\penalty-3 \space \@setcontribs
    \@closetoccontribs
  \fi
  \endtrivlist
  \endgroup
}
\def\@settitle{\begin{center}
\LARGE\lowercase{\@title}
  \end{center}%
}
\makeatother

\definecolor{lightblue}{HTML}{2B77A4}
\definecolor{darkred}{HTML}{9E0D0D}
\hypersetup{
	colorlinks=true,
	linkcolor=darkred,
	urlcolor=darkred,
	citecolor=lightblue
}
\title{Fourier decay of  product measures}

\author{Jonathan M. Fraser\\ \\
 University of St Andrews, Scotland\\
Email: jmf32@st-andrews.ac.uk}
\thanks{The  author was  financially supported by a  \emph{Leverhulme Trust Research Project Grant} (RPG-2019-034) and an \emph{EPSRC Standard Grant} (EP/Y029550/1).}

\begin{document}

%\date{}

\maketitle
\thispagestyle{empty}

\begin{abstract}
Can one characterise the Fourier decay of a product measure in terms of the Fourier decay of its marginals?  We make inroads on this question by describing  the Fourier spectrum of a product measure in terms of the Fourier spectrum of its marginals. The  Fourier  spectrum  is a continuously parametrised family  of dimensions living between the Fourier  and Hausdorff dimensions and captures more Fourier analytic information than either dimension considered in isolation.  We provide  several examples and applications, including to Kakeya and Furstenberg sets. In the process we derive   novel Fourier analytic characterisations of the upper and lower box dimensions.
\end{abstract}

\tableofcontents

\section{Introduction}

Given integers $1 \leq k < d$ and  finite Borel measures $\mu$  on $\mathbb{R}^k$ and  $\nu$   on $\mathbb{R}^{d-k}$, we are interested in the product measure $\mu \times \nu$  on $\rd$ and how it can be described in terms of its \emph{marginals}  $\mu$ and $\nu$.  A basic question from fractal geometry asks how the Hausdorff dimension of $\mu \times \nu$ depends on the Hausdorff dimensions of the marginals.  However, even this basic question is subtle.  One might naively expect the Hausdorff dimension of the product to be the sum of the Hausdorff dimensions of the marginals, but in general this is false and one cannot do better than
\begin{equation} \label{hdbe}
\hd \mu + \hd \nu \leq \hd \mu \times \nu \leq \hd \mu + \pd \nu
\end{equation} 
where $\hd$ dimension denotes Hausdorff dimension and $\pd$ denotes \emph{packing} dimension.  We shall revisit this phenomenon later, but for this paper the main question is about the Fourier analytic behaviour  of $\mu \times \nu$ and we are especially interested in Fourier decay.  Fourier decay is coarsely described  by the Fourier dimension, denoted by $\fd$, and it is easy to see that 
\begin{equation} \label{fdbe}
\fd (\mu \times \nu) = \min\{ \fd \mu , \, \fd \nu\}.
\end{equation} 
That is, the product cannot decay any better than its marginals.  However, this does not tell the full story.  In order to better understand the Fourier decay of a product measure, we consider the Fourier spectrum, recently introduced in \cite{fourierspec}, which is a  continuously parametrised family  of dimensions which interpolates between the Fourier  and Hausdorff/Sobolev  dimensions.  Even before we begin, it is clear that something non-trivial will happen, because the interpolation will witness the `Fourier behaviour' \eqref{fdbe} transforming into the `Hausdorff behaviour' \eqref{hdbe} as we vary the parameter; see our main results Theorems \ref{product} and \ref{product2}.  

As a by-product of our investigation, we are lead to  novel Fourier analytic characterisations of the upper and lower box dimensions of fractal sets, see Proposition \ref{boxxx}, which may have independent interest.  We give some simple applications in Section \ref{applications} and consider several simple examples in Section \ref{examples}, including surface measures on cylinders and spheres.  Then, in Section \ref{kakeyasection}, we use our results to construct Kakeya and Furstenberg sets with very bad Fourier analytic behaviour, for example, demonstrating that a result of Oberlin \cite{oberlin} does not generalise to higher dimensions.  

\subsection{Notation and convention}

Throughout the  paper, we write $A \lesssim B$ to mean there exists a constant $c >0$ such that $A \leq cB$.  The implicit constants $c$ are suppressed to improve exposition.  If we wish to emphasise that these constants depend on another parameter $\lambda$, then we will write $A \lesssim_\lambda B$.  We also write $A \gtrsim B$ if $B \lesssim A$ and $A \approx B$ if $A \lesssim B$ and $A \gtrsim B$.

We  write $a \wedge b = \min\{a,b\}$ and $a \vee b = \max\{a,b\}$ for real numbers $a$ and $b$. 

Finally, when we refer to a finite Borel measure, we exclude the zero measure to avoid trivial situations.

\subsection{The Fourier dimension and spectrum}

Let $\mu$ be a finite Borel measure on $\mathbb{R}^d$ with support denoted by $\textup{spt}(\mu)$.  
For $s \geq 0$, the $s$-energy of $\mu$ is  given by 
\[
\mathcal{I}_s(\mu) = \int \int \frac{d \mu(x) \, d \mu(y)}{|x-y|^s} 
\]
and can be used to estimate the Hausdorff dimension of $\mu$ and its support (the so-called `potential theoretic method').  Indeed, if $s \geq 0$ is such that $\mathcal{I}_s(\mu) <\infty$, then
\[
\hd \textup{spt}(\mu) \geq \hd \mu \geq s
\]
where $\hd$ denotes Hausdorff dimension. In fact, this is a precise characterisation of Hausdorff dimension for sets since for all Borel sets $X$ and $s < \hd X$, there exists a finite Borel measure $\mu$ on $X$ such that $\mathcal{I}_s(\mu) <\infty$.   See \cite{falconer, mattila} for more on Hausdorff dimension, energy and the potential theoretic method.

There is a useful connection between energy (and thus Hausdorff dimension) and the Fourier transform.   The \emph{Fourier transform} of $\mu$ is the function $\widehat \mu : \rd \to \mathbb{C}$ given  by
\[
\widehat \mu (z) = \int e^{-2\pi i z \cdot x} \, d \mu(x).
\]
In fact, for $0<s<d$ 
\[
\mathcal{I}_s(\mu) \approx_{s,d} \int_{\mathbb{R}^d} |\widehat \mu(z)  |^{2} |z|^{s-d} \, dz,
\]
see \cite[Theorem 3.1]{mattila}. Therefore, if $|\widehat \mu(z)  |\lesssim |z|^{-s/2}$ for some $s \in (0,d)$, then $\mathcal{I}_s(\mu) <\infty$ and $\hd \mu \geq s$.  This motivates the \emph{Fourier dimension}, defined by
\[
\fd \mu = \sup\{ s \geq 0 : |\widehat \mu(z)  |\lesssim |z|^{-s/2}\}
\]
for measures and
\[
\fd X =   \sup \{ \min\{ \fd \mu,  d\} : \textup{spt}(\mu) \subseteq X  \}
\]
for  sets $X \subseteq \rd$. Here the supremum is taken over finite Borel measures $\mu$ supported by   $X$.   For non-empty sets   $X \subseteq \rd$ 
\[
0 \leq \fd X \leq \hd X \leq d
\]
and $X$ is called a \emph{Salem set} if   $\fd X = \hd X$.  See \cite{mattila} for more on the Fourier dimension and \cite{grafakos} for   Fourier analysis more generally. There are many random constructions giving rise to Salem sets but non-trivial deterministic examples are much harder to come by---indeed, many prominent  examples fail to be Salem.  For example, a line segment in $\mathbb{R}$ is Salem but in $\mathbb{R}^2$ it is not.  Further, the middle third Cantor set, non-trivial  self-affine Bedford--McMullen carpets, the cone in $\mathbb{R}^3$ \cite{harris}, and the graph of fractional Brownian motion all fail to be Salem sets.

 In \cite{fourierspec} we introduced a continuum of dimensions (the Fourier spectrum) which vary continuously in-between the Hausdorff and Fourier dimensions.  The motivation behind this idea is to attempt to capture more Fourier analytic and geometric information than is provided by the Fourier and Hausdorff dimensions alone.  The hope is then to obtain  stronger applications in various contexts.  We have already had some success here, with the Fourier spectrum providing new information about dimensions of sumsets and convolutions \cite{fourierspec}, new estimates in the distance set problem \cite{fourierspec}, and new estimates for the dimension of the exceptional set in Marstrand's projection theorem \cite{ana2}. 

Following \cite{fourierspec}, for $\theta \in [0,1]$ and $s \geq 0$, we define energies
\[
\mathcal{J}_{s,\theta}(\mu) = \left( \int_{\mathbb{R}^d} |\widehat \mu(z)  |^{2/\theta} |z|^{s/\theta-d} \, dz \right)^{\theta},
\]
where we adopt the convention that
\[
\mathcal{J}_{s,0}(\mu) = \sup_{z \in \mathbb{R}^d} |\widehat \mu(z)  |^{2} |z|^{s}.
\]
We define the \emph{Fourier spectrum}  of $\mu$ at $\theta$ by 
\[
\fs \mu = \sup\{s \geq 0 : \mathcal{J}_{s,\theta}(\mu) < \infty\}
\]
where we write $\sup \o = 0$.  Note that 
\[
\mathcal{J}_{s,1}(\mu) = \int_{\mathbb{R}^d} |\widehat \mu(z)  |^{2} |z|^{s-d} \, dz 
\]
is  the familiar \emph{Sobolev energy} and, therefore, $\dim^1_{\mathrm{F}} \mu = \sd \mu$ where $\sd \mu$ is the \emph{Sobolev dimension} of $\mu$, see \cite[Section 5.2]{mattila}.  Moreover,  $\dim^0_{\mathrm{F}} \mu = \dim_{\mathrm{F}} \mu$ is the Fourier dimension of $\mu$. 

For   sets  $X \subseteq \mathbb{R}^d$, we define  the \emph{Fourier spectrum}    of $X$ at $\theta$ by 
\[
\fs X =   \sup \{ \min\{\fs \mu ,  d\}  :  \textup{spt}(\mu) \subseteq X  \}.
\]
Here the supremum is again taken over finite Borel measures $\mu$ supported by  $X$.  Evidently,  $\dim^0_{\mathrm{F}} X = \dim_{\mathrm{F}} X$ is the Fourier dimension of $X$ and, since   $\mathcal{J}_{s,1}(\mu)  \approx  \mathcal{I}_s(\mu)$ for $0<s<d$,   
 $\dim^1_{\mathrm{F}} X = \dim_{\mathrm{H}} X$ is  the Hausdorff dimension for Borel sets $X$.  It is easy to see  that
\[
\dim_{\mathrm{F}} \mu  \leq \fs \mu \leq \dim_{\mathrm{S}} \mu 
\]
and
\[
\dim_{\mathrm{F}} X  \leq \fs X \leq \dim_{\mathrm{H}} X
\]
for all $\theta \in (0,1)$.  As a function of $\theta$, $\fs\mu$ is continuous for $\theta\in(0,1]$ by \cite[Theorem~1.1]{fourierspec} and, in addition,  continuous at $\theta=0$ provided $\mu$ is compactly supported; see  \cite[Theorem~1.3]{fourierspec}.  Moreover, $\fs X$ is continuous for   $\theta\in[0,1]$ by \cite[Theorem~1.5]{fourierspec}; thus the Fourier spectrum gives a   continuous interpolation between the Fourier and Hausdorff dimensions for all Borel sets.

Strichartz \cite{stric1, stric2} considered bounds for  averages of the  Fourier transform of the form
\[
R^{d-\beta_k} \lesssim \int_{|z| \leq R} |\widehat \mu (z)|^{2k} \, dz \lesssim R^{d-\alpha_k}
\]
for integers $k \geq 1$ and $0 \leq \alpha_k \leq \beta_k$. Motivated by this, for $\theta \in (0,1]$, let
\[
\overline{F}_\mu(\theta) = \limsup_{R \to \infty} \frac{\theta \log \left(R^{-d} \int_{|z| \leq R} |\widehat \mu (z)|^{2/\theta} \, dz \right)}{-\log R}
\]
and
\[
\underline{F}_\mu(\theta) = \liminf_{R \to \infty} \frac{\theta \log \left(R^{-d} \int_{|z| \leq R} |\widehat \mu (z)|^{2/\theta}\, dz \right)}{-\log R}.
\]
In particular, $\overline{F}_\mu(\theta)$ is the infimum of $\beta \geq 0$ for which
\[
R^{d-\beta/\theta} \lesssim \int_{|z| \leq R} |\widehat \mu (z)|^{2/\theta} \, dz
\]
and  $\underline{F}_\mu(\theta)$ is the supremum  of $\alpha\geq 0$ for which
\[
  \int_{|z| \leq R} |\widehat \mu (z)|^{2/\theta} \, dz\lesssim R^{d-\alpha/\theta}.
\]
It is easy to see that 
\[
0 \leq \underline{F}_\mu(\theta)  \leq \overline{F}_\mu(\theta)  \leq d\theta
\]
and it was proved in \cite[Theorems 4.1--4.2]{fourierspec} that
\[
\underline{F}_\mu(\theta) = \min\{ d \theta , \, \fs \mu\}.
\]
In particular, the Fourier spectrum contains strictly more information than Strichartz' averaged Fourier dimensions in the case where the Fourier dimension is strictly positive.

\section{Main results: dimensions of product measures}

Our first result gives bounds  for the Fourier spectrum of a product measure in terms of the Fourier spectrum of its marginals.

\begin{thm} \label{product}
Let $1 \leq k<d$ be integers,  and $\mu$ and $\nu$ be    finite Borel measures on $\mathbb{R}^k$ and  $\mathbb{R}^{d-k}$, respectively.  Then
\[
\fs  (\mu \times \nu) \geq \min\Big\{ k\theta +\fs \nu ,  \ \fs \mu  + (d-k)\theta,  \  \fs \mu + \fs \nu\Big\}
\]
and 
\[
\fs  (\mu \times \nu) \leq \min\Big\{ k\theta +\fs \nu ,  \ \fs \mu  + (d-k)\theta \Big\}
\]
for all $\theta \in [0,1]$. In particular, $\fd  (\mu \times \nu) = \min \{  \fd \mu ,   \fd \nu   \}$.
\end{thm}

\begin{proof}
Throughout this proof we  write $z=(x,y) \in \mathbb{R}^d$ with $x \in \mathbb{R}^k$ and $y \in \mathbb{R}^{d-k}$.  Then
\[
\widehat{\mu \times \nu} (z) = \int e^{-2\pi i (x \cdot u + y \cdot v)} \, d \mu (u) \, d \nu (v) = \widehat \mu (x) \widehat \nu (y).
\]
This gives 
\[
|\widehat{\mu \times \nu} (z)| \lesssim \max\{ |\widehat \mu (x)|, \,  |\widehat \nu (y)|\}
\]
and by letting $|z| \to \infty$ keeping either $x$ or $y$ fixed equal to 0 shows that this weaker bound cannot be improved in general.  This yields
\[
\fd (\mu \times \nu) = \min\{ \fd \mu , \ \fd \nu\}
\]
which is the claim for $\theta = 0$.  From now on fix $\theta \in (0,1]$.

We begin with the lower bound. For this we  recall    \cite[Theorems 4.1--4.2]{fourierspec} which established that for a finite Borel measure $m$ on $\mathbb{R}^n$
\[
\underline{F}_m(\theta) = \min\{n \theta , \, \fs m\}.
\]
Therefore, for  $\eps>0$ which we now fix,
\begin{equation} \label{stricye}
\int_{\substack{u\in \mathbb{R}^{n}: \\ |u| \leq R} }  | \widehat m(u)  |^{2/\theta} \, dz \lesssim_\eps R^{n-n\wedge  \frac{\fs m}{ \theta} +\eps}.
\end{equation}
We will apply this estimate twice with   $m$ equal to $\mu$ and $\nu$, respectively.  First splitting $\rd$ into two half-spaces determined by the relative size of $x$ and $y$ and then applying \eqref{stricye}
\begin{align*}
\mathcal{J}_{s,\theta}( \mu \times \nu)^{1/\theta} &\leq \int_{\substack{z = (x,y)\in \mathbb{R}^k \times \mathbb{R}^{d-k}  : \\
|x| \geq |y|}}   | \widehat \mu(x)  |^{2/\theta}| \widehat \nu(y)  |^{2/\theta} |x|^{s/\theta-d} \, dz \\ 
& \qquad   \qquad   \qquad + \  \int_{\substack{z = (x,y)\in \mathbb{R}^k \times \mathbb{R}^{d-k}  : \\
|x| \leq |y|}}    | \widehat \mu(x)  |^{2/\theta}| \widehat \nu(y)  |^{2/\theta} |y|^{s/\theta-d} \, dz\\
&= \int_{ x\in \mathbb{R}^k }   | \widehat \mu(x)  |^{2/\theta}  |x|^{s/\theta-d} \int_{\substack{y\in \mathbb{R}^{d-k}: \\ |y| \leq |x| } } | \widehat \nu(y)  |^{2/\theta}  \, dy \, dx \\ 
& \qquad  \qquad   \qquad  + \  \int_{y \in \mathbb{R}^{d-k} }  |\widehat \nu(y)  |^{2/\theta} |y|^{s/\theta-d} \int_{\substack{x\in \mathbb{R}^{k}: \\ |x| \leq |y| } }  | \widehat \mu(x)  |^{2/\theta} \, dx \, dy \\
&\lesssim_\eps \int_{ x\in \mathbb{R}^k }   | \widehat \mu(x)  |^{2/\theta}  |x|^{s/\theta-d} |x|^{d-k-(d-k) \wedge \frac{\fs \nu}{\theta}+\eps} \, dx \\ 
& \qquad  \qquad   \qquad  + \  \int_{y \in \mathbb{R}^{d-k} }  |\widehat \nu(y)  |^{2/\theta} |y|^{s/\theta-d} |y|^{k-k \wedge \frac{\fs \mu }{\theta} +\eps} \, dy \\
&= \int_{ x\in \mathbb{R}^k }   | \widehat \mu(x)  |^{2/\theta}  |x|^{(s-(d-k)\theta \wedge \fs \nu +\eps \theta )/\theta-k}   \, dx \\ 
& \qquad  \qquad   \qquad  + \  \int_{y \in \mathbb{R}^{d-k} }  |\widehat \nu(y)  |^{2/\theta} |y|^{(s-k\theta \wedge \fs \mu+\eps \theta ) /\theta-(d-k)}   \, dy \\
& < \infty
\end{align*}
provided
\[
s-  (d-k)\theta \wedge \fs \nu +\eps \theta < \fs \mu 
\]
and
\[
s-k\theta \wedge \fs \mu+\eps \theta  < \fs \nu,
\]
that is, provided
\[
s < \min\Big\{ k\theta +\fs \nu ,  \ \fs \mu  + (d-k)\theta,  \  \fs \mu + \fs \nu\Big\} -\eps\theta.
\]
This proves the  desired lower bound upon letting $\eps \to 0$.

%not needed, but cool trick
\begin{comment}
 Writing 
\[
\{ y \in \mathbb{R}^{d-k} : |\widehat \nu (y)| >0\} = \bigcup_{n \in \mathbb{N}} \bigcup_{m \in \mathbb{N}} \left\{ y \in \mathbb{R}^{d-k} : |\widehat \nu (y)| >2^{-m} \text{ and } |y |\leq n\right\}
\]
we see that there must exist $m=m(\nu) \in (0,1)$ and a bounded Borel set $U=U(\nu) \subseteq \mathbb{R}^{d-k}$ with positive $(d-k)$-dimensional Lebesgue measure such that $|\widehat \nu (y)| > 2^{-m}$ for all $y \in U$.  Otherwise $\widehat \nu$ is the zero function almost everywhere and then $\widehat \nu \in L^1(\mathbb{R}^{d-k})$ and $\nu$ must be a continuous function, itself also in $L^1(\mathbb{R}^{d-k})$.  But then, by the Fourier inversion formula, $\nu$ must be the zero measure, which is a contradiction.
\end{comment}

We turn our attention to the upper bound.  Recall that $\widehat \nu $ is a continuous function (this follows easily by the dominated convergence theorem).  Also, $\widehat \nu (0) = \nu(\mathbb{R}^{d-k})>0$ and therefore there exists $\eps=\eps(\nu) \in (0,1)$ such that $|\widehat \nu (y)| \geq \nu(\mathbb{R}^{d-k})/2>0$ for all $y \in B_{d-k}(0,\eps)$, where $B_{d-k}$ denotes the    open ball  in $\mathbb{R}^{d-k}$.  Then
\begin{align*}
\mathcal{J}_{s,\theta}( \mu \times \nu)^{1/\theta} &\geq \int_{\mathbb{R}^k} \int_{\mathbb{R}^{d-k}}  | \widehat \mu(x)  |^{2/\theta}| \widehat \nu(y)  |^{2/\theta} (|x| \vee |y|)^{s/\theta-d}\, dy \, dx \\
& \gtrsim \int_{\mathbb{R}^k \setminus B_k(0,1)}  \int_{B_{d-k}(0,\eps)}  | \widehat \mu(x)  |^{2/\theta} |x|^{s/\theta-d} \,dy \, dx \\
& \approx \int_{\mathbb{R}^k \setminus B_k(0,1)} | \widehat \mu(x)  |^{2/\theta} |x|^{(s-(d-k)\theta)/\theta-k} \, dx \\
&=\infty
\end{align*}
whenever $s-(d-k)\theta > \fs \mu$, proving $\fs (\mu \times \nu) \leq \fs \mu + (d-k)\theta$. The symmetric bound $\fs (\mu \times \nu) \leq k \theta + \fs \nu$ may be obtained similarly.
\end{proof}

Theorem \ref{product} gives a precise formula for $\fs  (\mu \times \nu)$ in the case where either 
\[
\fs \mu \geq k \theta \qquad \text{ or } \qquad \fs \nu \geq (d-k) \theta.
\]
When neither of these conditions hold, we cannot expect a precise formula, recall \eqref{hdbe}.  However, in this case we can improve the upper bound by using Strichartz' upper averaged Fourier dimensions.

\begin{thm}\label{product2}
Let $1 \leq k<d$ be integers, and  $\mu$ and  $\nu$ be   finite Borel measures on  $\mathbb{R}^k$  and $\mathbb{R}^{d-k}$, respectively.  Then
\[
\fs \mu + \fs \nu \leq \fs  (\mu \times \nu) \leq  \max\Big\{ \overline{F}_\mu(\theta) +\fs \nu ,  \ \fs \mu  +\overline{F}_\nu(\theta) \Big\}
\]
for all $\theta \in [0,1]$.
\end{thm}

\begin{proof}
The lower bound comes from Theorem \ref{product} and is included for aesthetic reasons.  We now prove the upper bound.  For fixed  $\eps>0$,
\begin{align*}
\mathcal{J}_{s,\theta}( \mu \times \nu)^{1/\theta} &\gtrsim \int_{\substack{z = (x,y)\in \mathbb{R}^k \times \mathbb{R}^{d-k}  : \\
|x| \geq |y|}}   | \widehat \mu(x)  |^{2/\theta}| \widehat \nu(y)  |^{2/\theta} |x|^{s/\theta-d} \, dz \\ 
& \qquad   \qquad   \qquad + \  \int_{\substack{z = (x,y)\in \mathbb{R}^k \times \mathbb{R}^{d-k}  : \\
|x| \leq |y|}}    | \widehat \mu(x)  |^{2/\theta}| \widehat \nu(y)  |^{2/\theta} |y|^{s/\theta-d} \, dz\\
&\gtrsim  \int_{ x\in \mathbb{R}^k }   | \widehat \mu(x)  |^{2/\theta}  |x|^{s/\theta-d} \int_{\substack{y\in \mathbb{R}^{d-k}: \\ |y| \leq |x| } } | \widehat \nu(y)  |^{2/\theta}  \, dy \, dx \\ 
& \qquad  \qquad   \qquad  + \  \int_{y \in \mathbb{R}^{d-k} }  |\widehat \nu(y)  |^{2/\theta} |y|^{s/\theta-d} \int_{\substack{x\in \mathbb{R}^{k}: \\ |x| \leq |y| } }  | \widehat \mu(x)  |^{2/\theta} \, dx \, dy \\
&\gtrsim_\eps \int_{ x\in \mathbb{R}^k }   | \widehat \mu(x)  |^{2/\theta}  |x|^{s/\theta-d} |x|^{d-k-\overline{F}_\nu(\theta)/\theta-\eps} \, dx \\ 
& \qquad  \qquad   \qquad  + \  \int_{y \in \mathbb{R}^{d-k} }  |\widehat \nu(y)  |^{2/\theta} |y|^{s/\theta-d} |y|^{k-\overline{F}_\mu(\theta)/\theta - \eps} \, dy \\
&= \int_{ x\in \mathbb{R}^k }   | \widehat \mu(x)  |^{2/\theta}  |x|^{(s-\overline{F}_\nu(\theta)-\eps \theta )/\theta-k}   \, dx \\ 
& \qquad  \qquad   \qquad  + \  \int_{y \in \mathbb{R}^{d-k} }  |\widehat \nu(y)  |^{2/\theta} |y|^{(s-\overline{F}_\mu(\theta) - \eps \theta ) /\theta-(d-k)}   \, dy \\
& = \infty
\end{align*}
provided
\[
s-  \overline{F}_\nu(\theta) -\eps \theta > \fs \mu 
\]
or 
\[
s-\overline{F}_\mu(\theta) - \eps \theta  > \fs \nu,
\]
that is, provided
\[
s > \max\Big\{ \overline{F}_\mu(\theta) +\fs \nu ,  \ \fs \mu  +\overline{F}_\nu(\theta) \Big\} +\eps\theta.
\]
This proves the  desired upper bound upon letting $\eps \to 0$.
\end{proof}

\section{Dimensions of product sets}

Given the results of the previous section, it is natural to consider the Fourier spectrum of product sets.  We get the lower bounds for free by applying the results for measures but there is an extra subtlety in finding the upper bounds:~we do not know \emph{a priori} that the Fourier spectrum of a product set is realised by the Fourier spectrum of product measures. We get around this by disintegrating an arbitrary measure on the product.  Another subtlety in the case of sets is that we are naturally led to consider the \emph{extended Fourier spectrum} defined by
\[
\efs X =   \sup \{ \fs \mu   :  \textup{spt}(\mu) \subseteq X  \}
\]
and \emph{extended Fourier dimension} defined by
\[
\efd X =   \sup \{ \fd \mu   :  \textup{spt}(\mu) \subseteq X  \}.
\]
Note that if $X \subseteq \rd$ has Lebesgue measure zero, then $\efs X = \fs X$.  This is because  $\sd \mu>d$ implies that $\mu$ is supported on a set of positive measure; see \cite[Theorem 5.4]{mattila}.  Moreover, if $X$ has non-empty interior, then $\efs X = \efd X = \infty$ for all $\theta \in [0,1]$.  This is because if $X$ has non-empty interior, then it supports a measure with arbitrarily fast Fourier decay, e.g.~$f dz$ where $f$ is a Schwartz function supported   in the interior of $X$. In general $\fs X = \min\{\efs X,d\}$ and $\fd X = \min\{\efd X,d\}$.  Also, note that $\efd X, \efs X \in [0,2d] \cup\{\infty\}$ for all $\theta \in [0,1]$.  This is because if $X$ supports a measure $\mu$ with $\sd \mu >2d$ then $X$ has non-empty interior, in which case $\efd X = \infty$; see \cite[Theorem 5.4]{mattila}. 

 We do not know how to construct sets with extended Fourier dimension strictly exceeding $d$ but finite.  

\begin{ques}
Does there exist a  (compact) set $X \subseteq \rd$ with $d<\efs X \leq 2d$ for some $\theta \in [0,1]$? 
\end{ques}

The next theorem is our main result on the Fourier spectrum of product sets.

\begin{thm} \label{productset}
Let $1 \leq k<d$ be integers,  and $X \subseteq \mathbb{R}^k$ and  $Y \subseteq \mathbb{R}^{d-k}$ be compact sets.  Then
\[
\fs  (X \times Y) \geq \min\Big\{ k\theta +\efs Y ,  \ \efs X  + (d-k)\theta,  \  \efs X + \efs Y , \ d\Big\}
\]
and
\[
\fs  (X \times Y)  \leq \min\Big\{ k\theta +\efs Y ,  \ \efs X  + (d-k)\theta, \ d \Big\}
\]
for all $\theta \in [0,1]$.  In particular,  $\fd  (X \times Y) = \min \{  \efd X ,   \efd Y , d  \}$ and  if both $X$ and $Y$ have zero Lebesgue measure (in the appropriate ambient dimension), then $\fd  (X \times Y) = \min \{  \fd X ,   \fd Y   \}$.   
\end{thm}

\begin{proof}
The lower bound follows from Theorem \ref{product} and the definition of the Fourier spectrum of a set.  For the upper bound, let $m$ be an arbitrary Borel probability measure on $X \times Y$ and let $\mu$ denote the projection of $m$ onto $X$. Then, by the disintegration theorem, we may find a collection of  Borel probability measures $\{\nu_x : x \in X\}$ with each $\nu_x$ supported on $Y$ (not necessarily fully supported) and with the function $x \mapsto \nu_x$ Borel measurable such that  $m$ disintegrates as
\[
\int_{X \times Y} f(u,v) \, d m(u,v) = \int_X   \int_{Y} f(u,v) \, d \nu_u(v) \, d\mu(u)
\]
for $f \in L^1(m)$.  Therefore
\begin{align*}
\widehat{m} (x,y) &= \int_X   \int_{Y} e^{-2\pi i (x  \cdot u + y \cdot v)}  \,    d \nu_u(v) \, d\mu(u)  \\ 
&=\int_X    e^{-2\pi i x  \cdot u }  \widehat{\nu_u}(y) \,     d\mu(u)  
\end{align*}
and, in particular, 
\[
\widehat{m} (x,0) =\int_X    e^{-2\pi i x  \cdot u }  \widehat{\nu_u}(0) \,     d\mu(u) =  \widehat \mu (x)
\]
since $ \widehat{\nu_u}(0) = \nu_u(Y) = 1$ for all $u \in X$.   To relate the Fourier spectrum of $m$ to the Fourier spectrum of $\mu$ (and thus $X$),  we appeal to a discrete representation of the energies proved in \cite[Corollary 4.1]{ana}.  There it was shown that if $\nu$ is a finite Borel  measure supported on a compact set $A \subseteq \rd$ with diameter denoted by $|A|$, then for all $0<\alpha<1/|A|$
\[
\mathcal{J}_{s,\theta}( \nu)^{1/\theta}  \approx_\alpha |\widehat \nu (0)|^{2/\theta} + \sum_{z \in \alpha \mathbb{Z}^d \setminus\{0\}} |\widehat \nu(z)|^{2/\theta} |z|^{s/\theta-d}. 
\] 
In particular, since $X \times Y$ is compact, we can choose $0<\alpha<1/|X\times Y| \leq 1/|X|$ such that
\begin{align*}
\mathcal{J}_{s,\theta}( m)^{1/\theta} &  \approx_\alpha |\widehat m (0)|^{2/\theta} + \sum_{z \in \alpha \mathbb{Z}^d \setminus\{0\}} |\widehat m(z)|^{2/\theta} |z|^{s/\theta-d}\\
&\gtrsim |\widehat m (0,0)|^{2/\theta} + \sum_{x \in \alpha \mathbb{Z}^k \setminus\{0\}} |\widehat m(x,0)|^{2/\theta} |x|^{(s-(d-k)\theta)/\theta-k} \\
&= |\widehat \mu(0)|^{2/\theta} + \sum_{x \in \alpha \mathbb{Z}^k \setminus\{0\}} |\widehat \mu(x)|^{2/\theta} |x|^{(s-(d-k)\theta)/\theta-k} \\
&\approx_\alpha \mathcal{J}_{s-(d-k)\theta,\theta}( \mu)^{1/\theta} \\
&=\infty
\end{align*}
whenever $s-(d-k)\theta > \fs \mu$, proving $\fs m \leq  \fs \mu + (d-k)\theta $.  Since $m$ was arbitrary and $\mu$ is supported on $X$ this proves that $\fs (X \times Y ) \leq  \efs X+ (d-k)\theta $.  The analogous  upper bound $\fs (X \times Y)  \leq k\theta + \efs Y$ follows by a symmetric argument. Finally, the upper bound of $d$ holds since $X \times Y \subseteq \rd$.
\end{proof}

\section{A Fourier analytic characterisation of  box dimension}

The fact that the Hausdorff dimension of a Borel set can be expressed in terms of the Sobolev dimension of measures it supports is well-known and goes back to the potential theoretic method and Frostman's lemma.  By considering how the Sobolev dimension relates to the Fourier spectrum and to Strichartz' averaged Fourier dimensions, this leads to various Fourier analytic characterisations of the Hausdorff dimension.  We record some of these here for convenience; see \cite{falconer, mattila, fourierspec} for the details.

\begin{prop}\label{boxdimthm}
Let $X \subseteq \rd$ be a Borel set.  Then
\begin{align*}
\hd X &= \dim_\textup{F}^1  X \\
&=\sup \{ \min\{d, \sd \mu\} :  \textup{spt}(\mu) \subseteq X  \}\\
&=\sup \{ \underline{F}_\mu(1) :  \textup{spt}(\mu) \subseteq X  \}\\
&=\sup\left\{ \beta \in  [0,d] : \exists \mu \text{ on $X$ such that } \mathcal{J}_{\beta,1}(\mu) < \infty \right\} \\
&=\sup\left\{ \beta \geq 0 : \exists \mu \text{ on $X$ such that } \int_{|z| \leq R} |\widehat \mu (z)|^{2} \, dz \lesssim_\mu  R^{d-\beta}   \right\}.
\end{align*}
\end{prop}

Given the above, it is natural to ask what happens if we consider appropriate duals of the above characterisations of Hausdorff dimension, for example, replacing  the lower version of Strichartz' averaged Fourier dimensions $\underline{F}_\mu(1)$ with the upper $\overline{F}_\mu(1)$.  Our work on dimensions of product measures gives some hints in this direction. Recall that
\[
\hd (X \times Y) \leq \hd X + \textup{Dim} \,  Y
\]
for bounded sets $X,Y$ in Euclidean space and where  $\textup{Dim}$ can be taken to be  either the upper box dimension or the packing dimension, see \cite[Chapter 7]{falconer} and also recall \eqref{hdbe}.  Given   how the upper and lower versions of Strichartz' averaged Fourier dimensions are connected to the dimensions of products, see Theorem \ref{product2}, it is natural to expect that dualising some of the characterisations  above  should lead to either the upper box dimension or the packing dimension.  We provide one such result here which gives a novel Fourier analytic characterisation of the upper box dimension. By dualising our result we obtain a similar characterisation of the lower box dimension. Our proof will rely on recent work of Falconer connecting upper box dimension and capacities \cite{falconerfourier}.  We briefly recall the definition of the upper and lower box dimensions $\ubd X$ and $\lbd X$ for  a bounded set $X \subseteq \rd$; see \cite{falconer} for more details.  Given a small scale $r>0$, let $N_r(X)$ denote the minimum number of sets of diameter $r$ that can cover $X$.  Then
\[
\ubd X = \inf \{ \beta \geq 0 : N_r(X) \lesssim r^{-\beta}\}
\]
and
\[
\lbd X = \sup \{ \beta \geq 0 : N_r(X) \gtrsim r^{-\beta}\}.
\]

\begin{prop} \label{boxxx}
Let $X \subseteq \rd$ be a bounded set.  Then
\begin{align*}
\ubd X &=\inf\left\{ \beta \geq 0 : \inf_{ \textup{$\mu$ on $X$}} \int_{|z| \leq R} |\widehat \mu (z)|^{2} \, dz \gtrsim  R^{d-\beta}   \right\}
\end{align*}
and
\begin{align*}
\lbd X &=\sup\left\{ \beta \geq 0 : \inf_{ \textup{$\mu$ on $X$}} \int_{|z| \leq R} |\widehat \mu (z)|^{2} \, dz \lesssim  R^{d-\beta}   \right\}.
\end{align*}
\end{prop}

\begin{proof}
%We offer two proofs of this result, simply to connect the ideas to  different aspects of the literature. First we show how to derive it from recent work of Falconer connecting upper box dimension and capacities \cite{falconerfourier}. 
We first give the proof of the upper box dimension result. Following \cite[(4) and (6)]{falconerfourier}, for a scale $r\in (0,1)$ and a `dimension' $s \in (0,d)$, define the \emph{capacity} $C_r^s(X)$ of $X$ by 
\[
\frac{1}{C_r^s(X)} = \inf\left\{ \iint \min\left\{1, \left(\frac{r}{|x-y|} \right)^s \right\} \, d \mu(x) \, d\mu(y) : \textup{spt}(\mu) \subseteq X \right\}
\]
and the \emph{upper $s$-dimensional box dimension profile} of $X$ by
\[
\ubdp X = \inf \{ \beta \geq 0 : C_r^s(X)  \lesssim r^{-\beta}\}.
\]
It follows from  \cite[Proposition 5.1]{falconerfourier} that 
\begin{equation} \label{profilego}
\ubdp X \nearrow  \ubd X
\end{equation}
as $s \to d$.  (Set $t=n$,  apply \cite[Corollary 2.5]{falconerfourier}, and let $s \to t=n$.)  Let $\eps>0$ and consider a particular Borel probability measure $\mu$ on $X$.  First applying  \cite[Lemma 3.2 and Proposition 3.3]{falconerfourier}, and then writing $R=1/r$,
\begin{align*}
 &\, \hspace{-1cm} \iint \min\left\{1, \left(\frac{r}{|x-y|} \right)^s \right\} \, d \mu(x) \, d\mu(y) \\
 &\approx   r^s \int  |z|^{s-d} \exp \left( -\frac{1}{2} r^2 |z|^2 \right)   |\widehat \mu(z)|^2 \, dz\\
&\approx R^{-s}   \int  |z|^{s-d} \exp \left( -\frac{1}{2}  (|z|/R)^2 \right)   |\widehat \mu(z)|^2 \, dz \\
&\lesssim  R^{-s}   \int_{|z| \leq R^{1+\eps}}  |z|^{s-d}    |\widehat \mu(z)|^2 \, dz  + R^{-s}   \int_{|z| > R^{1+\eps}}  R^{s-d} \exp \left( -\frac{1}{2}  (|z|/R)^2 \right)  dz \\ 
&\lesssim  R^{-s}   \int_{|z| \leq R^{1+\eps}}  |z|^{s-d}    |\widehat \mu(z)|^2 \, dz  + R^{-d}   \int_{|y| > R^\eps}    \exp \left( -\frac{1}{2}  y^2 \right)  R^d dy \\ 
&\lesssim R^{-s}   \int_{|z| \leq R^{1+\eps}}    |\widehat \mu(z)|^2 \, dz  + \frac{O(R^{-\eps})}{\sqrt{\exp(R^{2\eps})}}\\
&\lesssim_\eps R^{-s}   \int_{|z| \leq R^{1+\eps}}    |\widehat \mu(z)|^2 \, dz. \numberthis \label{goodupper}
\end{align*}
In the above we used standard asymptotic properties of the Gauss error function  to deal with the tail.  Similarly,
\begin{align*}
 \iint \min\left\{1, \left(\frac{r}{|x-y|} \right)^s \right\} \, d \mu(x) \, d\mu(y) &\gtrsim   R^{-s}   \int_{|z| \leq R}  R^{s-d}  |\widehat \mu(z)|^2 \, dz  \\
&= R^{-d}   \int_{|z| \leq R}    |\widehat \mu(z)|^2 \, dz \numberthis \label{goodlower}
\end{align*}
noting that the implicit constants above can be taken independent of $\mu$.   Therefore, by \eqref{goodupper}
\begin{align*}
\ubdp X &\geq \inf \left\{ \beta \geq 0 :  \inf_{ \textup{$\mu$ on $X$}} \int_{|z| \leq R^{1+\eps}} |\widehat \mu (z)|^{2} \, dz \gtrsim  R^{s-\beta}   \right\} \\
& \geq (1+\eps) \inf \left\{ \beta \geq 0 :   \inf_{ \textup{$\mu$ on $X$}} \int_{|z| \leq R} |\widehat \mu (z)|^{2} \, dz \gtrsim  R^{d-\beta}   \right\} +s-d(1+\eps)
\end{align*}
and by \eqref{goodlower}
\[
\ubdp X \leq \inf \left\{ \beta \geq 0 :   \inf_{ \textup{$\mu$ on $X$}} \int_{|z| \leq R} |\widehat \mu (z)|^{2} \, dz \gtrsim  R^{d-\beta}   \right\}.
\]
Letting $\eps \to 0$ and then $s \to d$, the desired result    follows from these estimates and \eqref{profilego}.

The lower box dimension result is handled in a completely  analogous way and we omit the details.  
\end{proof}

It remains an interesting question to determine the significance of 
\[
\sup \{ \overline{F}_\mu(1) :  \textup{spt}(\mu) \subseteq X  \} = \inf\left\{ \beta \geq 0 :  \textup{ for all $\mu$ on $X$ } \int_{|z| \leq R} |\widehat \mu (z)|^{2} \, dz \gtrsim_\mu  R^{d-\beta}   \right\}
\]
as a function of $X$ but we have been unable to do this.  That said, it is clear that
\[
\sup \{ \overline{F}_\mu(1) :  \textup{spt}(\mu) \subseteq X  \} \leq \inf\left\{ \beta \geq 0 : \inf_{ \textup{$\mu$ on $X$}} \int_{|z| \leq R} |\widehat \mu (z)|^{2} \, dz \gtrsim  R^{d-\beta}   \right\}.
\]
The potential difference here is that the right hand side asks for the  implicit constants to be independent of $\mu$ whereas the left hand side allows the constants to depend on $\mu$.  In particular, using Propositions \ref{boxdimthm} and \ref{boxxx},
\[
\hd X \leq \sup \{ \overline{F}_\mu(1) :  \textup{spt}(\mu) \subseteq X  \} \leq \ubd X.
\]
We are tempted by the following conjecture, which we have been unable to prove. Recall we write $\pd X$ for the packing dimension of $X$.
\begin{conj}
Let $X \subseteq \rd$ be a Borel set.  Then
\begin{align*}
\pd X &= \sup \{ \overline{F}_\mu(1) :  \textup{spt}(\mu) \subseteq X  \}\\
 &= \inf\left\{ \beta \geq 0 :  \textup{ for all $\mu$ on $X$ } \int_{|z| \leq R} |\widehat \mu (z)|^{2} \, dz \gtrsim_\mu  R^{d-\beta}   \right\}.
\end{align*}
\end{conj}
Note that in the above conjecture, the second equality is immediate from the definition of $\overline{F}_\mu(1)$ and so only the first inequality is unknown.

\begin{comment}
We draw one more simple connection using work on the box dimensions of measures.  The upper box dimension of a compactly supported measure was defined in \cite{mink} to be 
\[
\ubd \mu = \inf\left\{ \beta \geq 0 :\text{   for all $x \in \textup{spt}(\mu)$ and $r \in (0,1)$ } \mu(B(x,R) \gtrsim r^\beta  \right\}.
\]
and it was proved in  \cite[Theorem 2.1]{mink} that if $X$ is compact then
\begin{align*}
\ubd X &= \inf\left\{\ubd \mu  :   \mu \text{ fully supported on $X$ } \right\}.
\end{align*}
Furthermore,   if $\mu$ satisfies $\mu(B(x,R) \gtrsim r^\beta$, then
\[
\int_{|z| \leq R} |\widehat \mu (z)|^{2} \, dz \gtrsim  R^{d-\beta},
\]
see \cite[page 45]{mattila}.  Therefore
\[
 \overline{F}_\mu(1) \leq \ubd \mu.
\]
 \end{comment}

%only gives lower bound?
\begin{comment}
We now offer an alternative proof, which uses our work with Falconer and K\"aenm\"aki on box dimensions of measures \cite{mink} together with an estimate from \cite{mattila}.  First 
\end{comment}

\section{Further applications} \label{applications}

What has emerged above is that product structure is bad for Fourier decay.  To demonstrate this further,  consider the question of when a product set can be a Salem set.  This is certainly  possible and trivial examples include zero dimensional sets $\{(0,0)\}$ and full dimensional sets $[0,1]^2$.  We show in the following corollary that these extreme examples are the only possibilities.  

\begin{cor} 
Let $1 \leq k<d$ be integers, and  $X$ and  $Y$ be     Borel sets  in  $\mathbb{R}^k$  and $\mathbb{R}^{d-k}$, respectively.  If  $\hd (X \times Y) \in (0,d)$, then $X \times Y$ is not a Salem set.
\end{cor}
\begin{proof}
Suppose  $X \times Y$ is   a Salem set.  Applying Theorem \ref{productset}, we get
\begin{align*}
d > \hd (X \times Y) 
&= \fd  (X \times Y) \\
&=\min \{  \efd X ,   \efd Y  \}\\
&=  \hd  (X \times Y)  \\
&\geq  \hd X + \hd Y \\
&\geq  \min \{  \efd X ,  k  \} +  \min \{  \efd Y ,  (d-k)  \}.
\end{align*}
Write $t = \min \{  \efd X ,   \efd Y  \}$ and by assumption $t \in (0,d)$.  By the above, we also have
\[
t \geq   t \wedge  k  +   t \wedge (d-k).
\]
Since $t>0$ this forces $t \geq k \vee (d-k)$  and we get $t \geq k+(d-k) = d$, a contradiction. 
\end{proof}

Let $\mu$ be a finite Borel measure on $\rd$ and consider the $n$-fold product measure $ \mu^n = \mu \times \cdots \times \mu$ on $\mathbb{R}^{nd}$.  Then perhaps one expects the dimension of $ \mu^n$ to increase with $n$, perhaps even with $\fs  \big( \mu^n \big) \approx n \fs \mu$.  That is, the asymptotic growth rate of the dimension of $\mu^n$ is dictated by the dimension of $\mu$.  Indeed, if $\fs \mu < d \theta$, then by Theorem \ref{product}
\[
\fs \big( \mu^n \big) \geq n \fs \mu
\]
and this is an equality provided $\overline{F}_\mu(\theta) = \fs \mu$; see  Theorem \ref{product2}.  However, we see that there is a relative loss of smoothness for measures with large dimension:~if $\fs \mu$ is  large, the asymptotic growth of $\fs \big( \mu^n \big) $ does not see this. The following is an immediate consequence of Theorem \ref{product}.
\begin{cor}
Let $\mu$ be a finite Borel measure on $\rd$ and suppose $\theta \in [0,1]$ is such that $d \theta \leq \fs \mu < \infty$.  Then
\[
\fs \big( \mu^n \big) = (n-1) d \theta + \fs \mu
\]
for all integers $ n\geq 1$.  In particular, if $\theta=0$, then $\fd \big( \mu^n \big) =   \fd \mu \lesssim 1$ and if $\theta >0$ then   $\fs \big( \mu^n \big) \sim n d \theta$ as $n \to \infty$.
\end{cor}

\begin{comment}
phase transition, 
\begin{cor}
Let $X$ be a   Borel set on $\rd$ with $0<   \fd X  \leq \hd X <d$.  Then there must exist a unique $\theta_0 \in (0,1)$ such that $\dim^{\theta_0}_{\mathrm{F}} = d\theta_0$ and $\fs  (X \times X  )$ must have a point of non-differentiability $\theta_0$.  
\end{cor}

\begin{proof}
The existence and uniqueness of $\theta_0$ follows by the fact that $\fs X$ is continuous and has Lipschitz constant at most $d$ by \cite{ana}. Then, by Theorem \ref{product},
\[
\fs  (X \times X  ) = d\theta+ \fs X
\]
for $\theta \leq \theta_0$ and
\[
\fs  (X \times X  ) \geq  2\fs X
\]
for $\theta\geq \theta_0$.  In particular, the upper left derivative of $\fs  (X \times X  ) $ is  $d$ plus the upper    left derivative of $\fs  X $ and the lower right derivative of  $\fs  (X \times X  ) $  is at least twice the lower right derivative of  $\fs   X  $.
\end{proof}
\end{comment}

\section{Examples: surface measures and cylinders} \label{examples}

Lebesgue measure is a curious example as we will see.  Writing $\mathcal{L}^d$ for the Lebesgue measure on $[0,1]^d \subseteq \mathbb{R}^d$, it is well-known and easy to show that
\[
|\widehat{\mathcal{L}^1} (z) | \lesssim |z|^{-1}
\]
and that the exponent 1 is best possible. Therefore $\fd \mathcal{L}^1 = 2$.  Moreover, it is easy to see that
\[
 \fs \mathcal{L}^1 = 2
\]
for all $\theta \in [0,1]$.   One might expect the Fourier dimension of $\mathcal{L}^d$  to increase with $d$, but this is not the case.  
\begin{prop} \label{lebesgue}
For all $d \geq 1$
\[
\fs  \mathcal{L}^d = 2+(d -1)\theta
\]
for all $\theta \in [0,1]$.  In particular, $\fd \mathcal{L}^d  = 2$ and $\sd \mathcal{L}^d  = d+1$  for all $d \geq 1$.
\end{prop}

\begin{proof}
The claim for $d=1$ is simple,  and mentioned above. For $d \geq 2$, assume the claim is true for $d-1$ and proceed by induction. Since $\mathcal{L}^d = \mathcal{L}^1 \times \mathcal{L}^{d-1}$ (as a Borel measure), applying Theorem \ref{product}
\begin{align*}
\fs  \mathcal{L}^d &= \min \Big\{ \theta +2+(d -2)\theta, \ 2+(d-1)\theta ,  \ 2+ 2+(d -2)\theta\Big\}  = 2+(d -1)\theta
\end{align*}
as required.
\end{proof}

Despite the somewhat counter-intuitive nature of the previous result, it is true that
\begin{equation} \label{fscube}
\fs [0,1]^d = d
\end{equation}
for all $d\geq 1$ and all $\theta \in [0,1]$.  The (perhaps) surprising thing is that this is not witnessed by Lebesgue measure, but rather by, for example, $fdz$ for a Schwartz function $f$ with support contained in $(0,1)^d$.

Another simple but important  example is the sphere $S^k \subseteq \mathbb{R}^{k+1}$ which is well-known to be a Salem set.  Moreover, this is witnessed by the surface measure, see \cite{mattila}.  Indeed, writing $\sigma^k$ for the surface measure on $S^k$,
\[
\fs S^k = \fs \sigma^k = k
\]
for all $\theta \in [0,1]$.

Let $ k,n \geq 1$ be integers and consider the cylinder $C_{k,n} = S^{k} \times [0,1]^{n} \subseteq \mathbb{R}^{k+1+n}$.  Note that $C_{1,1}$ is the familiar cylinder $S_1 \times [0,1]$ in $\mathbb{R}^3$.  In general  $C_{k,n}$ is a smooth $(k+n)$-dimensional manifold (with a boundary) and the surface measure is $\sigma^k \times \mathcal{L}^n$ where $\sigma^k$ is the surface measure on $S^k$ and  $\mathcal{L}^n$  is the Lebesgue measure on $[0,1]^{n}$.  
\begin{cor} \label{cylinder}
For all  integers $ k,n \geq 1$,
\[
\fs  (\sigma^k \times \mathcal{L}^n)  =   \min\Big\{ 2+(k+n)\theta ,  \ k  + n\theta \Big\}
\]
for all $\theta \in [0,1]$.  This has a phase transition at $\theta = 1-2/k$ whenever $k\geq 3$.
\end{cor}
\begin{proof}
  The surface measure $\sigma^k$ on $S^{k} \subseteq \mathbb{R}^{k+1}$ is a Salem measure with dimension  $k$ and so $\fs S^{k}= k$ for all $\theta$. Proposition \ref{lebesgue} gives that $\fs \mathcal{L}^n = 2+(n-1)\theta$.   It now follows from Theorem \ref{product} that
\begin{align*}
\fs  (\sigma^k \times \mathcal{L}^n)  &\geq \min\Big\{ (k+1) \theta + \fs \mathcal{L}^n,  \ \fs \sigma^k  + n\theta ,  \  \fs \sigma^k + \fs \mathcal{L}^n\Big\}\\
  &\geq \min\Big\{ (k+1) \theta + 2+(n-1)\theta ,  \ k  + n\theta ,  \  k + 2+(n -1)\theta\Big\}\\
& =   \min\Big\{ 2+(k+n)\theta ,  \ k  + n\theta \Big\}
\end{align*}
and
\begin{align*}
\fs  (\sigma^k \times \mathcal{L}^n)   &\leq  \min\Big\{ (k+1) \theta + \fs \mathcal{L}^n,  \ \fs \sigma^k  + n\theta \Big\}\\ 
&\leq \min\Big\{ (k+1) \theta + 2+(n-1)\theta ,  \ k  + n\theta  \Big\}\\
& =   \min\Big\{ 2+(k+n)\theta ,  \ k  + n\theta \Big\}
\end{align*}
as required.
\end{proof}
%
\begin{comment}
Some notable special cases of the above include:
\[
\fs  (\sigma^1 \times \mathcal{L}^1) = 1+\theta 
\]
\[
\fs  (\sigma^k \times \mathcal{L}^1) = \min\Big\{ 2+ (k+1)\theta   ,  \ k  + \theta  \Big\}
\]
\[
\fs  (\sigma^k \times \mathcal{L}^2) = \min\Big\{2+ (k+2)\theta ,  \ k  +  2\theta\Big\}
\]
\[
\fs  (\sigma^1 \times \mathcal{L}^n)  =   \min\Big\{ 2+(1+n)\theta ,  \ 1  + n\theta \Big\} = 1+n\theta.
\]
\end{comment}
The situation for the cylinder \emph{sets} $C_{k,n}$ is simpler.
\begin{cor}
For all  integers $ k,n \geq 1$,
\[
\fs C_{k,n} =  \fs S^{k}+ n\theta = k+n\theta.
\]
\end{cor}
\begin{proof}
By Theorem \ref{productset}, 
\begin{align*}
&\, \hspace{-2mm} \fs C_{k,n} \\
 &\geq \min\Big\{ \efs S^k+n\theta, \ (k+1)\theta + \efs [0,1]^n, \ \efs S^k+  \efs [0,1]^n , \ k+1+n\Big\}\\
&=  \min\Big\{ k+n\theta, \ (k+1)\theta + \infty, \ k+\infty, \  k+1+n\Big\}\\
&=k+n\theta
\end{align*}
and
\begin{align*}
\fs C_{k,n}  &\leq \min\Big\{ \efs S^k+n\theta, \ (k+1)\theta + \efs [0,1]^n, \ k+1+n\Big\}\\
&=  \min\Big\{ k+n\theta, \ (k+1)\theta + \infty, \ k+1+n\Big\}\\
&=k+n\theta,
\end{align*}
as required. 
\end{proof}

\begin{figure}[H]
							\includegraphics[width=\textwidth]{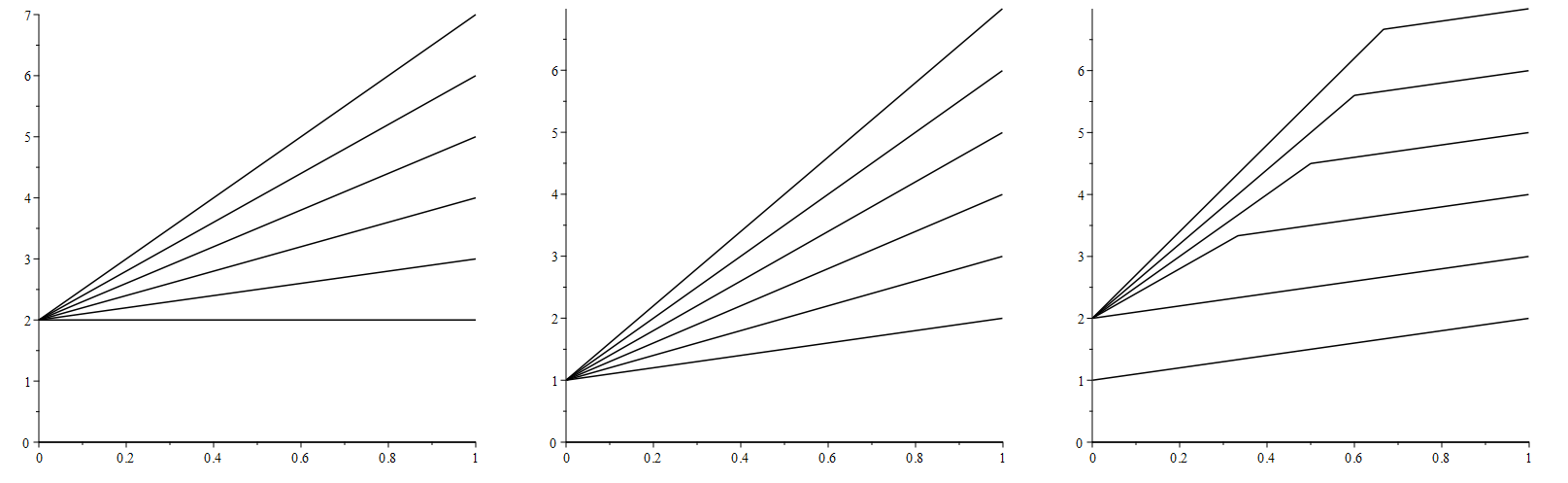} \\ \vspace{3mm}

			\caption{\emph{Three examples.}  Left: the Fourier  spectrum of $d$-dimensional Lebesgue measure on the unit cube for $d=1,2,\dots, 6$; see Proposition \ref{lebesgue}.  Centre: the Fourier spectrum of the surface measure on the cylinder $S^1 \times [0,1]^n$ for $n=1, \dots, 6$; see Corollary \ref{cylinder}.  Right: the Fourier spectrum of the surface measure on the cylinder $S^k \times [0,1]$ for $k=1, \dots, 6$; see Corollary \ref{cylinder}. These three families of measures are not distinguished by Sobolev dimension and, within each family, are not distinguished by Fourier dimension for $k \geq 2$.  However, the Fourier spectrum distinguishes all 17 distinct  measures considered apart from in one case:   $[0,1]^2$ and $S^2 \times [0,1]$.  Note here that $S^1 \times [0,1]$ and $[0,1] \times S^1$ are the same measure upon relabelling of coordinates.
}
			\end{figure}

\section{Kakeya and Furstenberg type sets} \label{kakeyasection}

A  \textit{Kakeya set} in $\mathbb{R}^d$ is a set which contains a unit line segment in every possible direction.  The famous \emph{Kakeya conjecture} is that such sets must have Hausdorff dimension $d$. This is despite the fact that they can have zero $d$-dimensional Lebesgue measure for $d \geq 2$; a result of Besicovitch \cite{besicovitch}.  Davies proved that Kakeya sets in $\mathbb{R}^2$ must have Hausdorff dimension 2 \cite{davies}, thus resolving the conjecture in the plane.  The problem   is open for $d \geq 3$ but various partial results are known; see \cite{katz, katztao} and  also \cite[Figure 5]{hickman} for a summary of the state of the art as of 2019.

In \cite{oberlin} Oberlin proved that compact Kakeya sets in $\mathbb{R}^2$ have Fourier dimension 2.  This beautiful result implies (and is rather stronger than) Davies' result mentioned above.  In \cite{harris} we  generalised Oberlin's argument to prove that compact Kakeya sets in $\mathbb{R}^d$ have Fourier dimension at least 2. While the Kakeya conjecture (for Hausdorff dimension) remains open, we observe that the Fourier dimension analogue cannot be extended any further than Oberlin's result.  That is, there are Kakeya sets in $\rd$ with Fourier dimension 2 for all $d \geq 3$.  This answers a question we raised in \cite[Question 7.2]{harris} in the negative.

\begin{prop} \label{kak}
For all integers $d \geq 3$, there exists a compact Kakeya set $K$ in $\rd$ such that
\[
\fd K = 2.
\]
\end{prop}

\begin{proof}
 Let $K_0 \subseteq \mathbb{R}^2$ be a compact Kakeya set with zero 2-dimensional Lebesgue measure.  The fact that  such Kakeya sets exist goes back to Besicovitch \cite{besicovitch}; see also \cite[Theorem 11.1]{mattila}.  Then $K = K_0 \times [0,1]^{d-2}$ is a compact Kakeya set in $\rd$ and Theorem \ref{productset} gives that
\[
\fd K = \min \{  \efd K_0 ,   \efd ([0,1]^{d-2}) , d  \} = \min\{2, \infty, d\} = 2
\]
where we crucially use the fact that $K_0$ has 2-dimensional Lebesgue measure zero and so $\efd K_0 = \fd K_0 = 2$.
%old proof
\begin{comment}
Let $\mu$ be a finite Borel measure on $K$ and $\pi : \rd \to \mathbb{R}^2$ denote orthogonal projection onto the first 2 coordinates.  Then $\pi(\mu)$ is a finite Borel measure on $K_0$.  Let $\eps>0$.  Since $K_0$ has zero Lebesgue measure, there must exist a sequence $x_k \in \mathbb{R}^2$ with $|x_k| \to \infty$ such that 
\[
|\widehat{\pi(\mu)}(x_k) | \gtrsim |x_k|^{-1-\eps}.
\]
Otherwise, $\widehat{\pi(\mu)}$ would be in $L^2$ and the support of $\pi(\mu)$ would have positive Lebesgue measure.  Then, writing $z_k = x_k \times 0 \in \rd$
\[
|\widehat \mu (z_k) |= |\widehat{\pi(\mu)}(x_k) | \gtrsim |x_k|^{-1-\eps} = |z_k|^{-1-\eps} 
\]
and $\fd \mu \leq 2(1+\eps)$.  It follows that $\fd K \leq 2$.
\end{comment}
\end{proof}

Given $s,t >0$, an $(s,t)$-\emph{Furstenberg set} $U$ in $\mathbb{R}^d$ is a set for which there is a set of lines $\mathcal{L}$ of Hausdorff dimension at least $t$ such that for every line $L \in \mathcal{L}$ , $\hd U \cap L \geq s$.  Note that the space of lines in $\rd$ has dimension $2(d-1)$ and so we only consider $0 < t \leq 2(d-1)$ and $0 < s \leq 1$. In particular, Kakeya sets are $(1,d-1)$-Furstenberg sets.  There is a lot of interest in proving lower bounds for the Hausdorff dimension of $(s,t)$-Furstenberg sets in terms of $s$ and $t$; see significant recent papers  \cite{GSW19,OS23,OS23+,wang} and the references therein. The planar case was recently resolved by Ren and Wang \cite{wang} where it was proved that an $(s,t)$-Furstenberg set $U$ in $\mathbb{R}^2$ has Hausdorff dimension at least
\[
\min\left\{s+t, \frac{3s+t}{2}, s+1\right\}.
\]
In particular, this  bound is the best possible.  Recalling Oberlin's work on Kakeya sets it is natural to ask about the Fourier dimension of   $(s,t)$-Furstenberg sets  but we observe that nothing can be said in general.  That is, for all $d \geq 2$, there exist $(1,2(d-1))$-Furstenberg sets with Fourier dimension 0.  In fact we prove the following stronger statement.

\begin{prop} \label{furst}
For all integers $d \geq 2$, there exists a compact set $U \subseteq \rd$ such that there is a   set of lines $\mathcal{L}$ with  non-empty interior in the $2(d-1)$-dimensional manifold of all lines such that $U$ intersects every line in $\mathcal{L}$  in a set of positive length, but yet
\[
\fd U = 0.
\]
\end{prop}

\begin{proof}
Let $U_0 \subseteq [0,1]$ be a compact set with positive 1-dimensional Lebesgue measure but $\fd U_0 = 0$.  Such sets were shown to exist by Ekstr\"om, Persson and  Schmeling; see \cite[Example 7]{modfourier}.  Let   $U = U_0 \times [0,1]^{d-1}$ observing that $U$ is a compact subset of $ [0,1]^d$.

 Let $\pi$ denote projection from $\rd$ onto the line $\mathbb{R} \times \{0_{d-1}\}$ where $0_{d-1}$ is the origin in $\mathbb{R}^{d-1}$  and let $\mathcal{L}$ be the set of lines $L$ in $\rd$  with the property that 
\[
\pi(L \cap [0,1]^d) = \pi(  [0,1]^d)  = [0,1]\times \{0_{d-1}\}.
\]
Then, for every $L \in \mathcal{L}$,   $L \cap U$ contains an affine  image of $U_0$ and thus has positive length. Moreover, the collection of lines $\mathcal{L}$ clearly has non-empty interior in the space of all lines; indeed, it contains a neighbourhood of the line
\[
\mathbb{R} \times \{0_{d-1}\} + (0, 1/2, 1/2, \cdots, 1/2).
\]
By Theorem \ref{productset}
\[
\fd U = \min \{  \efd U_0 ,   \efd ([0,1]^{d-1}) , d  \} = \min\{0, \infty, d\} = 0
\]
where we crucially use the fact that $U_0$ has $\fd U_0 = 0$ and so $\efd U_0 = \fd U_0 = 0$.
%old proof
\begin{comment}
Let $\mu$ be a finite Borel measure on $F$ and $\pi : \rd \to \mathbb{R}$ denote orthogonal projection onto the first coordinate.  Then $\pi(\mu)$ is a finite Borel measure on $F_0$.  Let $\eps>0$.  Since $F_0$ has zero Fourier dimension, there must exist a sequence $x_k \in \mathbb{R}$ with $|x_k| \to \infty$ such that 
\[
|\widehat{\pi(\mu)}(x_k) | \gtrsim |x_k|^{-\eps}.
\]
Otherwise, $\fd F_0 \geq  \fd  \pi(\mu) >0$.  Then, writing $z_k = x_k \times 0 \in \rd$
\[
|\widehat \mu (z_k) |= |\widehat{\pi(\mu)}(x_k) | \gtrsim |x_k|^{-\eps} = |z_k|^{-1-\eps} 
\]
and $\fd \mu \leq 2\eps$.  It follows that $\fd F =0$.
\end{comment}
\end{proof}

Apart from putting an end to Fourier analytic analogues of the Kakeya conjecture or questions about Furstenberg sets, we view Propositions \ref{kak} and \ref{furst} as highlighting the elegance of Oberlin's result   that Kakeya sets in $\mathbb{R}^2$ have Fourier dimension 2.  In particular, this is sharp in the sense that it cannot be pushed further in either of the above considered natural directions.

\section*{Acknowledgements}

The author thanks Kenneth Falconer for helpful discussions.

\end{document}